\documentclass[11pt,a4paper]{article}
\usepackage[english]{babel}
\usepackage{amsmath}
\usepackage{amssymb,amsbsy}
\usepackage{amsthm}
\usepackage[utf8]{inputenc}
\usepackage[T1]{fontenc}
\usepackage{enumerate}
\usepackage{mathrsfs}
\usepackage{bbm}
\usepackage{bm}
\usepackage{color}
\usepackage{hyperref}
\usepackage{authblk}

\usepackage{nicefrac}

\newtheorem{tw}{Theorem}[section]
\newtheorem{lm}[tw]{Lemma}

\newtheorem{pr}[tw]{Proposition}

\theoremstyle{definition}
\newtheorem{df}{Definition}[section]

\DeclareMathOperator{\diam}{diam}

\newcommand{\R}{\mathbb{R}}
\newcommand{\Z}{\mathbb{Z}}
\newcommand{\N}{\mathbb{N}}
\newcommand{\T}{\mathbb{T}}

\newcommand{\cQ}{\mathcal{Q}}

\newcommand{\cB}{\mathcal{B}}
\newcommand{\cC}{\mathcal{C}}

\newcommand{\cP}{\mathcal{P}}
\newcommand{\cT}{\mathcal{T}}

\newcommand{\cI}{\mathcal{I}}

\newcommand{\1}{\mathbbm{1}}

\newcommand{\bea}{\begin{eqnarray}}
  \newcommand{\eea}{\end{eqnarray}}
  \newcommand{\beab}{\begin{eqnarray*}}
  \newcommand{\eeab}{\end{eqnarray*}}

  \newcommand{\be}{\begin{equation}}
  \newcommand{\ee}{\end{equation}}
  
\providecommand{\noopsort}[1]{} 

\title{On rank of von Neumann special flows}
\author{Adam Kanigowski, Anton Solomko}

\begin{document}
\maketitle

\begin{abstract}
We prove that special flows over an ergodic rotation of the circle under a $C^1$ roof function with one discontinuity do not have local rank one.
In particular, any such flow has infinite rank.
\end{abstract}

\section{Introduction}

Special flows built over the rotation $R_\alpha$ of the circle by an irrational number $\alpha$ and under a piecewise $C^1$-function $f$ with non-zero sum of jumps were introduced and studied by J.~von~Neumann in \cite{vN32}.
He proved that such flows have continuous spectrum (i.e.\ are weakly mixing) for each irrational rotation.
The weak mixing result was generalised in two directions:
the $C^1$-condition on $f$ was replaced by absolute continuity in \cite{ILM99},
while A.~Katok, \cite{Ka01}, proved weak mixing replacing $R_\alpha$ by any ergodic interval exchange transformations.
Recall also that it follows from a result by A.~Ko\v{c}ergin, \cite{Ko72}, that von Neumann flows are never mixing.
As a matter of fact K.~Fr\c{a}czek and M.~Lema\'{n}czyk, \cite{FL04}, proved that von Neumann flows are spectrally disjoint from all mixing flows.
Moreover, whenever $\alpha$ has bounded type, von Neumann flows are mildly mixing, \cite{FL06}.

On the other hand, nothing is known about spectral multiplicity and rank properties of von Neumann flows. Recall that rank yields an upper bound on the spectral multiplicity \cite{Ki}.

In the paper we show that the rank of von Neumann flows for which $f$ has one discontinuity is infinite.  
More precisely, let $\cT=(T_t)_{t\in\R}$ be a von Neumann special flow over the rotation $R_\alpha\colon\T\to\T$ by an irrational $\alpha$ and under a roof function $f\colon \T\to\R_+$ of the form
\begin{equation}\label{roof}
f(x)=g(x)+A\{x\}+c,
\end{equation}
where $g\in C^1(\T)$, $\int_\T g(x)dx=0$, $A\neq 0$, $c\in\R$ is such that $f>0$ and $\{x\}$ stands for the fractional part of $x$.
Our main result is the following:
\begin{tw}\label{mainth}
The flow $\cT$ does not have local rank one.
\end{tw}

The key property which is used to prove Theorem \ref{mainth} is {\em slow divergence} of orbits of nearby points in the flow direction.
Such way of divergence of orbits is characteristic for {\em parabolic systems}.
It was first observed by M.~Ratner for the class of horocycle flows, \cite{Ra83}.
Later this property (called Ratner's property) was shown to hold for some von Neumann flows under the additional assumption that $\alpha$ has bounded type \cite{FL06} and for some mixing, smooth flows on surfaces \cite{FK}, \cite{KKU}.
In all these papers, Ratner's property was used to enhance mixing properties (weak mixing to mild mixing and mixing to multiple mixing).
In \cite{K16}, a property of Ratner's type, called in \cite{K16} {\em parabolic divergence}, was used to compute {\em slow entropy} of some  mixing smooth surface flows with non-degenerate saddles.
A variant of the parabolic divergence property is the main ingredient in the proof of Theorem \ref{mainth}.

Let us shortly describe this property in our setting. Notice that the divergence of orbits of two close points is caused by two effects: slow (uniform) divergence by hitting the roof and fast (non-uniform) divergence by hitting the discontinuity. It is the first type of divergence which is characteristic for parabolic systems. It turns out that either going forward or backward in time, we can always avoid the discontinuity long enough to observe the effect of the uniform divergence (see Lemma~\ref{lin} and Proposition~\ref{mainprop}).

\section{Basic definitions}

We denote by $\T$ the circle group $\R/\Z$ which we will identify with the unit interval $[0,1)$.
For a real number $x$ denote by $\{x\}$ its fractional part, by $[x]=x-\{x\}$ its integer part and by $\|x\|$ its distance to the nearest integer.
Given $x,y\in\T$ with $\|x-y\|<\frac{1}{2}$, $[x,y]$ will denote the shortest interval in $\T$ connecting $x$ and $y$.
Lebesgue measure on $\T$ will be denoted by $m_\T$.



\subsection{Special flows}

Let $T$ be an ergodic automorphism of a standard Borel space $(X,\cB,\mu)$ (with $\mu(X)<+\infty$).
A measurable function $f\colon X\to \R$ defines a cocycle $\Z\times X \to \R$ given by
$$
f^{(n)}(x) =
\begin{cases}
f(x)+f(Tx)+\cdots+f(T^{n-1}x) &\mbox{if}\quad n>0 \\
\hfil 0 &\mbox{if}\quad n=0 \\
\hfil -(f(T^{n}x)+\cdots+f(T^{-1}x)) &\mbox{if}\quad n<0.
\end{cases}
$$
Assume that $f\in L^1(X,\cB,\mu)$ is a strictly positive function.

\begin{df}\label{specialflow}
The \emph{special flow} $\cT = (T^f_t)_{t\in\R}$ over the \emph{base automorphism} $T$ under the \emph{roof function} $f$ is the flow acting on $(X^f,\cB^f,\mu^f)$, where $X^f = \{ (x,s) \in X\times\R \mid 0 \leq s < f(x) \}$ and $\cB^f$ and $\mu^f$ are the restrictions of $\cB \otimes \cB_\R$ and $\mu \otimes \lambda$ to $X^f$ respectively ($\lambda$ stands for Lebesgue measure on $\R$).
Under the action of the flow $\cT$ each point in $X^f$ moves vertically upward with unit speed, and we identify the point $(x,f(x))$ with $(Tx,0)$.
More precisely, for $(x,s)\in X^f$ we have
\begin{equation}\label{spflow}
T^f_t(x,s) = (T^n x, s+t-f^{(n)}(x)),
\end{equation}
where $n\in\Z$ is the unique number such that $f^{(n)}(x) \leq s+t < f^{(n+1)}(x)$.
\end{df}
\noindent
We will identify the base $X$ with the subset $\{ (x,0) \mid x\in X\}\subset X^f$.
Notice that if $d$ is a 
 metric on $X$, then $d^f((x,t),(y,s)) = d(x,y) + |t-s|$ is a metric for $X^f$.

\subsection{Von Neumann flows}

We call a function $f\colon \T \to \R$ \emph{piecewise $C^1$} if there exist $\beta_1,\ldots,\beta_k\in\T$ such that $f|_{\T\setminus \{\beta_1,\ldots,\beta_k\}}$ is $C^1$ smooth and $f_\pm(\beta)=\lim_{x\to \beta\pm} f(x)$ is assumed to be finite.
Denote $d_i:=f_-(\beta_i)-f_+(\beta_i)$ the \emph{jump} of $f$ at point $\beta_i$.
The number $\sum_{i=1}^k d_i$ is the \emph{sum of jumps} of $f$.

\begin{df}
A \emph{von Neumann flow} is a special flow $\cT$ over a rotation $R_\alpha\colon(\T,m_\T)\to(\T,m_\T)$ by an irrational $\alpha\in\T$ and under a piecewise $C^1$ roof function $f\colon \T\to\R_+$ with a non-zero sum of jumps.
\end{df}

We will consider the simplest case when $f$ has only one discontinuity.
Without loss of generality we may assume that $f$ is $C^1$ on $\T\setminus \{0\}$ with a jump $A=f_-(0)-f_+(0)\neq 0$.
Any such $f$ can be written in the form
$$
f(x)=g(x)+A\{x\}+c,
$$
where $g\in C^1(\T)$, $\int_\T g(x)dx=0$, $A\neq 0$ ($A$ is called the \emph{slope}) and $c\in\R$ is such that $f>0$.

We will assume that $\int_\T f dm_\T=1$, that is we normalize the resulting measure to make it a probability measure.

\subsection{Finite rank systems and systems of local rank one}

In this section we recall the notion of finite rank and local rank one.
There are several equivalent ways to define a finite rank system (see \cite{Fe97}).
We will define rank properties in the language of special flows, \cite{Fa05}.

Let $\cT=(T_t)_{t\in\R}$ be an ergodic flow on a standard probability space $(X,\cB,\mu)$.
Let $B\subset X$, $H\in\R_+$ and $0<\eta < 1$.
\begin{df}\label{tower}
A pair $(B,H)$ is an \emph{$\eta$-tower} (or simply a \emph{tower}) for $\cT$ (with the \emph{base} $B$ and of \emph{height} $H$) if there exists an isomorphism $\varphi\colon (X,\mu) \to (Y^f,\nu^f)$ between the flow $\cT$ and a special flow over an ergodic $S\colon (Y,\cC,\nu) \to (Y,\cC,\nu)$ under a roof function $f\colon Y \to \R_+$ such that
\begin{itemize}
\item $\varphi(B) \subset Y$ and $\nu(\varphi(B))>\eta \nu(Y)$;
\item $f(y)\leq H$ for every $y\in Y$;	
\item $f(y)=H$ for every $y\in \varphi(B)$.
\end{itemize}
\end{df}
\noindent
By abuse of notation, we will identify the subset $\varphi^{-1}(Y)\subset X$ with the base $Y$ and the measure $\nu\circ \varphi^{-1}$ on $\varphi^{-1}(Y)$ with $\nu$, that is we will think of the base of the corresponding suspension flow as of a subset of $X$ equipped with a finite measure $\nu$.
Note that $\nu(Y)\geq\frac{1}{H}$ (since $1=\mu(X)=\nu^f(Y^f)\leq H\nu(Y)$).
The set $\bigsqcup_{t=0}^H T_t B$ will be also referred to as a \emph{tower}.
For $0\leq t<H$ the set $T_t B$ is called a {\em level} of the tower.




Fix a finite measurable partition $\cP$ of $X$.
For $\varepsilon > 0$ we say that a level $T_t B$ is {\em $\varepsilon$- monochromatic}
(for $\cP$) if a $1-\varepsilon$ proportion of it (with respect to the measure $\nu\circ T_{-t}$)
is contained in one atom of the partition $\cP$.
A tower for $(T_t)_{t\in\R}$ is called {\em $\varepsilon$-monochromatic} (for $\cP$) if a $1-\varepsilon$ proportion of its levels (with respect to the Lebesgue measure $\lambda$ on $[0,H)$) is $\varepsilon$-monochromatic.

\begin{df}\label{alpharankone}
Let $\beta\in(0,1]$.
An ergodic flow $\cT=(T_t)_{t\in\R}$ has \emph{local rank one of order $\beta$} if for every finite partition $\cP$ of $X$ and every $\varepsilon>0$ there exists a $(\beta-\varepsilon)$-tower for $\cT$ which is $\varepsilon$-monochromatic for $\cP$.
\end{df}

\begin{df}\label{localrankone}
An ergodic flow $\cT$ has {\em local rank one} if it has local rank one of some order $\beta\in(0,1]$.
If $\beta=1$ then $\cT$ is said to have {\em rank one}.
\end{df}

\begin{df}\label{finiterank}
An ergodic flow $\cT=(T_t)_{t\in\R}$ has \emph{finite rank} if there exists $r\in\N$ such that for every finite partition $\cP$ of $X$ and every $\varepsilon>0$ there are $r$ disjoint towers $(B_1,H_1),\ldots,(B_r,H_r)$ for $\cT$ which are $\varepsilon$-monochromatic for $\cP$ and such that $\mu(\bigsqcup_{i=1}^r \bigsqcup_{t=0}^H T_t B_i) > 1-\varepsilon$.
We get the definition of rank one flow if $r=1$.
\end{df}

\noindent
Clearly, all the above properties are measure theoretic invariants and the following implications hold:
$$
\mbox{rank one} \; \Rightarrow \; \mbox{finite rank} \; \Rightarrow \; \mbox{local rank one} \; \stackrel{\scriptsize\cite{Ki}}{\Rightarrow} \; \mbox{finite spectral multiplicity}.
$$

Let $\cP\colon X \to \{1,\ldots,k\}$ be a finite partition of $X$.
Given $H>0$ we define the \emph{Hamming distance} between $x$ and $y\in X$ by
$$
d^\cP_H(x,y) := \frac{1}{H} \lambda(\{ 0\leq t <H \mid \cP(T_tx) \neq \cP(T_ty) \}).
$$

\begin{lm}\label{hamlem1}
For an arbitrary $\varepsilon>0$, if $(B,H)$ is an $\frac{\varepsilon^2}{4}$-monochromatic tower for a flow $(T_t)_{t\in\R}$, then there exists $D\subset B$, $\nu(D)>(1-\varepsilon)\nu(B)$, such that $d^\cP_{H}(x,y)<\varepsilon$ for all $x,y\in D$.
\end{lm}

\begin{proof}
For every $0\leq t<H$, the partition $\cP$ induces a partition $\cQ_t$ of $B$ by $\cQ_t(x)=\cP(T_t x)$.
Denote by $Q_t$ the atom of the partition $\cQ_t$ with the largest $\nu$-measure.
If for some $t\in [0,H)$, $Q_t$ is not uniquely defined, take $Q_t$ to be any atom of $\cQ_t$ for which $\nu$ obtains its maximum.
However, for $1-\frac{\varepsilon^2}{4}$ proportion of $t\in [0,H)$ (with respect to the Lebesgue measure) $Q_t$ is determined uniquely, since the tower is $\frac{\varepsilon^2}{4}$-monochromatic.
Denote by $P_t$ the atom of $\cP$ such that $Q_t=B\cap T_{-t}(P_t)$.
Note that
$$
\frac{1}{H}\int_0^{H} \left( \int_B \1_{B\setminus Q_t}(x) d\nu(x) \right) dt < \frac{\varepsilon^2}{2},
$$
since the tower is $\frac{\varepsilon^2}{4}$-monochromatic.
By Fubini's Theorem, we get
$$
\frac{1}{H}\int_B \left( \int_0^{H} \1_{B\setminus Q_t}(x) dt \right) d\nu(x) < \frac{\varepsilon^2}{2}.
$$
Hence there exists a set $D\subset B$, $\nu(D) > (1-\varepsilon)\nu(B)$, such that for all $x\in D$,
$$
\frac{1}{H} \int_0^{H} \1_{B\setminus Q_t}(x) dt < \frac{\varepsilon}{2}.
$$
This means that for every $x\in D$, $T_tx\in P_t$ for $1-\frac{\varepsilon}{2}$ proportion of $t\in[0,H)$.
So, for any $x,y\in D$ and for $1-\varepsilon$ proportion of $t\in [0,H)$ both $T_tx,T_ty\in P_t$, and therefore $d^\cP_{H}(x,y)<\varepsilon$.
\end{proof}

Given a set $P$ in a metric space $(X,d)$, by its \emph{diameter} we mean $\diam P := \sup_{x,y\in P}d(x,y)$.
For a family of sets $\cP$, $\diam \cP := \sup_{P\in\cP}\diam P$.
As we will see in the next section, von Neumann flows in consideration have a Ratner like property of slow divergence of nearby points (Proposition~\ref{mainprop}), which will be in contrast with the following property of local rank one special flows.

\begin{lm} \label{lem2}
Let $\cT=(T_t)_{t\in\R}$ be a special flow over a circle rotation and let $\cP$ be a finite partition of $X$.
If $\cT$ has local rank one of order $\beta\in(0,1]$, then for any $\varepsilon > 0$ there exist $H=H(\varepsilon)>0$ arbitrary large and $(x_0,s_0),(y_0,s'_0)\in X$ such that
$$\|x_0-y_0\| > \frac{\beta}{10H}, \;
d^f((x_0,s_0),(y_0,s'_0)) \leq \diam \cP, \;
d^\cP_H((x_0,s_0),(y_0,s'_0))<\varepsilon,$$
and for $(x_1,s_1):=T_H(x_0,s_0)$, $(y_1,s'_1):=T_H(y_0,s'_0)$ we also have
$$\|x_1-y_1\| > \frac{\beta}{10H}, \quad
d^f((x_1,s_1),(y_1,s'_1)) \leq \diam \cP.$$
\end{lm}


\begin{proof}
We assume that $\varepsilon<\frac{\beta}{20}$.
Since $\cT$ has local rank one of order $\beta$, we can find an arbitrary high $(\beta-\varepsilon)$-tower $(B,h)$ which is $\frac{\varepsilon^2}{16}$-monochromatic for $\cP$, $\nu(B)>(\beta-\varepsilon)\nu(Y)\geq \frac{\beta-\varepsilon}{h}$ (recall that $\nu(Y)\geq \frac{1}{h}$).
Fix $t_0\in[0, h/4)$ and $t_1\in[3h/4, h)$ for which the levels $T_{t_0} B$ and $T_{t_1} B$ are $\varepsilon$-monochromatic, i.e.\ $\nu(T_{-t_i}P_i\cap B)\geq (1-\varepsilon)\nu(B)$ for an atom $P_i$ of $\cP$ ($i=0,1$).
We set $H:=t_1-t_0 > \frac{h}{2}$.

By Lemma~\ref{hamlem1} there exists a subset $D\subset B$ with $\nu(D)>(1-\varepsilon)\nu(B)$ and such that $d^\cP_h((x,s),(y,s'))<\frac{\varepsilon}{2}$ for any $(x,s),(y,s')\in D$.
Let $E:=D \cap T_{-t_0}P_0 \cap T_{-t_1}P_1$, $\nu(E)>(1-3\varepsilon)\nu(B) > \frac{\beta-4\varepsilon}{h} > \frac{\beta}{2.5H}$.
Fix some $(x,s)\in E$ and set
$$(x_0,s_0):=T_{t_0}(x,s), \quad (x_1,s_1):=T_{t_1}(x,s) = T_H(x_0,s_0).$$
Similarly, for $(y,s')\in E$ denote $$(y_0,s'_0):=T_{t_0}(y,s'), \quad (y_1,s'_1):=T_{t_1}(y,s') = T_H(y_0,s'_0).$$
Consider two sets
\begin{align*}
C_0 &:=\{ (y,s')\in E \mid \|x_0-y_0\|\leq \tfrac{\beta}{10H} \} \quad\mbox{and} \\
C_1 &:=\{ (y,s')\in E \mid \|x_1-y_1\|\leq \tfrac{\beta}{10H} \}.
\end{align*}
It is clear that $\nu(C_0)=\nu(C_1)= \frac{\beta}{5H}$ (since $\nu|_B$ projects onto $m_\T|_{\pi(B)}$, where $\pi$ stands for the natural projection $X\to\T$), while $\nu(E)>\frac{\beta}{2.5H}$.
Therefore there exists $(y,s') \in E \setminus (C_0 \cup C_1)$.
The corresponding pair $(x_0,s_0)$ and $(y_0,s'_0)=T_{t_0}(y,s')$ satisfies the lemma.
Indeed, $\|x_0-y_0\|, \|x_1-y_1\| > \frac{\beta}{10H}$, because $(y,s') \notin C_0 \cup C_1$.
On the other hand,
$$
d^f((x_i,s_i),(y_i,s'_i)) \leq \diam \cP,
$$
because $(x_i,s_i),(y_i,s'_i)\in P_i$, $i=0,1$.
And finally,
\begin{align*}
d^\cP_H((x_0,s_0),(y_0,s'_0))
&= H^{-1} \lambda(\{ t_0\leq t <t_1 \mid \cP(T_t(x,s)) \neq \cP(T_t(y,s')) \}) \\
&\leq H^{-1} \lambda(\{ 0\leq t <h \mid \cP(T_t(x,s)) \neq \cP(T_t(y,s')) \}) \\
&< 2h^{-1} \lambda(\{ 0\leq t <h \mid \cP(T_t(x,s)) \neq \cP(T_t(y,s')) \}) \\
&= 2d^\cP_h((x,s),(y,s')) < \varepsilon,
\end{align*}
since $H>\frac{h}{2}$ and $(x,s),(y,s')\in D$.
The lemma is proved.
\end{proof}

\section{Proof of Theorem~\ref{mainth}}
Assume that an irrational $\alpha$ is fixed and $f$ is of the form \eqref{roof}.
We may assume without lost of generality that the slope in \eqref{roof} is $A=1$.
Denote $L := \min\{1,\inf_{x\in\T} f(x)\}$ and $M := \max\{1,\sup_{x\in\T} f(x)\}$, so that
$$
0<L\leq f(x) \leq M < \infty
$$
and $L\leq 1 \leq M$.
For an integer interval $J=[a,b]\subset \Z$ and $c>0$, by $cJ$ we denote the integer interval $[ca,cb]\cap\Z$.
The following proposition will be the main ingredient in the proof.

\begin{pr} \label{mainprop}
Let $\beta\in (0,1]$ be a fixed number.
There exists a constant $\delta_0>0$ such that for all $x,y\in \T$, $0<\|x-y\|<\delta_0$, there exists an interval $J_{x,y}\subset \Z$ such that the following  holds:
\begin{enumerate}[(i)]
\item \label{prop1a} $0\in J_{x,y}\subset [-\frac{\beta}{100M\|x-y\|},
\frac{\beta}{100M\|x-y\|}]$;
\item \label{prop1b} for every $n\in \frac{10M}{L} J_{x,y}$,   we have $|f^{(n)}(x)-f^{(n)}(y)|<\frac{50M}{L}$;
\item \label{prop1c} there exists an interval $U_{x,y}\subset J_{x,y}$ such that $|U_{x,y}|>\frac{1}{10}|J_{x,y}|$ and such that for every $n\in U_{x,y}$
$$
|f^{(n)}(x)-f^{(n)}(y)|>\frac{\beta L}{10^6M^2}.
$$
\end{enumerate}
\end{pr}

We will prove the proposition in the next section, now let us prove the main result.

\begin{proof}[Proof of Theorem~\ref{mainth}]
Assume, by contradiction, that $\cT$ has local rank one of order $\beta\in(0,1]$.
Fix a positive
$$
\delta < \min\Bigl\{ \delta_0,\;
\frac12 \frac{\beta L}{10^6M^2},\;
\min_{\substack{k\in\Z\\ 0<k\leq \frac{60M}{L^2}}} \frac{\|k\alpha\|}{2} \Bigr\}, 
$$
where $\delta_0$ is as in Proposition~\ref{mainprop}.
Let $\cP$ be a finite partition of $X$ of diameter less then $\delta$.
For an arbitrary $\varepsilon > 0$ we can find by Lemma~\ref{lem2} a positive $H>100M$
and a pair of points $(x_0,s_0),(y_0,s'_0)\in D$ such that, if we denote by $(x_t,s_t):=T_t(x_0,s_0)$ and $(y_t,s'_t):=T_t(y_0,s'_0)$, we have
\begin{equation}\label{eq1}
\|x_0-y_0\| > \frac{\beta}{10H}, \quad \|x_H-y_H\| > \frac{\beta}{10H},
\end{equation}
\begin{equation}\label{eq2}
d^f((x_0,s_0),(y_0,s'_0))<\delta, \quad d^f((x_H,s_H),(y_H,s'_H))<\delta,
\end{equation}
\begin{equation}\label{hamineq}
d^\cP_{H}((x_0,s_0),(y_0,s'_0))<\varepsilon.
\end{equation}
Denote
\begin{align*}
C &:=\{ 0\leq t < H \mid d^f((x_t,s_t),(y_t,s'_t)) < \delta \} \quad\mbox{and} \\
F &:=\{ 0\leq t < H \mid d^f((x_t,s_t),(y_t,s'_t)) \geq \delta \}
\end{align*}
the times when the points $(x_t,s_t)$ and $(y_t,s'_t)$ are $\delta$-close and $\delta$-far respectively.
Notice that the mapping $\R\ni t\mapsto d((x_t,s_t),(y_t,s'_t))$ is piecewise constant by the definition of a special flow, whence both $C$ and $F$ are finite unions of (half-open) intervals.
It trivially follows from the definition of $F$ (since $\diam\cP<\delta$) that $\lambda(F)/H \leq d^\cP_{H}((x_0,s_0),(y_0,s'_0))$.
We will show at the end that $\lambda(F)>\varepsilon H$ if $\varepsilon$ is small to get a contradiction with (\ref{hamineq}).
For the reader's convenience, we split the proof into several steps.
Given $t\in C$, let $J_{x_t,y_t}$ be as in Proposition~\ref{mainprop}.

\smallskip

\textit{Claim~1.}
\textit{For every $t\in C$, $|J_{x_t,y_t}|\leq \frac{H}{5M}$.}

To prove the claim we will argue by contradiction.
Assume that there exists $t_0\in C$ such that $|J_{x_{t_0},y_{t_0}}| > \frac{H}{5M}$.
Then we derive from Proposition~\ref{mainprop}\eqref{prop1a} that
\begin{equation}\label{spk}
\|x_{t_0}-y_{t_0}\| < \frac{\beta}{10H}.
\end{equation}
Set $J:=\frac{10M}{L} J_{x_{t_0},y_{t_0}}$, so that $|J|>\frac{2H}L$ and for every $n\in J$
\begin{equation}\label{pb}
|f^{(n)}(x_{t_0})-f^{(n)}(y_{t_0})|< \frac{50M}{L}
\end{equation}
(by Proposition~\ref{mainprop}\eqref{prop1b}).
Consider the interval $I\subset \R$ given by
$$I:=[t_0-s_{t_0}+f^{(a)}(x_{t_0}), t_0-s_{t_0}+f^{(b)}(x_{t_0}) ),$$
where $a=\min J$, $b=\max J+1$.
Then $\lambda(I)>2H$ and it follows that either $0 \in I$ or $H \in I$.
Since the proof in both cases goes along the same lines, let us assume without loss of generality that $H \in I$.
By definition,
\begin{align*}
&(x_H,s_H) = T_{H-t_0}(x_{t_0},s_{t_0}) = (x_{t_0}+n\alpha, s_{t_0}+(H-t_0)-f^{(n)}(x_{t_0})), \\
&(y_H,s'_H) = T_{H-t_0}(y_{t_0},s'_{t_0}) = (y_{t_0}+m\alpha, s'_{t_0}+(H-t_0)-f^{(m)}(y_{t_0})),
\end{align*}
where $n,m\in\Z$ are such that $f^{(n)}(x_{t_0}) \leq s_{t_0}+(H-t_0) < f^{(n+1)}(x_{t_0})$, $f^{(m)}(y_{t_0}) \leq s'_{t_0}+(H-t_0) < f^{(m+1)}(y_{t_0})$.
Notice that $n\in J$ by the definition of $I$.
The distance between $(x_H,s_H)$ and $(y_H,s'_H)$ is smaller than $\delta$ by \eqref{eq2} and equals
\begin{equation}\label{dist}
\|x_{t_0}-y_{t_0}+(n-m)\alpha\| + |s_{t_0}-s'_{t_0}-f^{(n)}(x_{t_0})+f^{(m)}(y_{t_0})| < \delta.
\end{equation}
Since $f<M$, $|s_{t_0}-s'_{t_0}|<M$.
It follows from \eqref{pb}, \eqref{dist} and the cocycle equality $f^{(m)}(y_{t_0}) = f^{(n)}(y_{t_0}) + f^{(m-n)}(y_{t_0}+n\alpha)$ that
$$|f^{(m-n)}(y_{t_0}+n\alpha)|<\frac{50M}{L}+M+1<\frac{60M}{L}.$$
On the other hand, $f(x)\geq L$ and therefore $|f^{(k)}(x)| \geq |k|L$ for any $x\in\T$, $k\in\Z$.
Hence $|n-m|<\frac{60M}{L^2}$.
Since the first term in (\ref{dist}) is less then $\delta$ and $\|x_{t_0}-y_{t_0}\|<\delta$ (because $t_0\in C$), we have $\|(n-m)\alpha\|<2\delta$ with $|n-m|<\frac{60M}{L^2}$.
By the choice of $\delta$ this is only possible if $n=m$.
We conclude that $x_H=x_{t_0}+n\alpha$, $y_H=y_{t_0}+n\alpha$ and 
$$
\|x_H-y_H\|=\|x_{t_0}-y_{t_0}\|\stackrel{\eqref{spk}}{<} \frac{\beta}{10H},
$$
which contradicts \eqref{eq1}.
Similarly, if $0\in I$ then one can show that $\|x_0-y_0\|=\|x_{t_0}-y_{t_0}\| < \frac{\beta}{10H}$.
Claim~1 is proved.

\smallskip

\textit{Claim~2.}
\textit{For any $t\in C\cap[\frac{H}{4},\frac{3H}{4}]$ there exists an interval $I_t\subset [0,H)$ containing $t$
such that $\lambda(I_t\cap F) > \gamma \lambda(I_t)$, where $\gamma = \frac{L}{10M}$.}

Fix $t\in C\cap[\frac{H}{4},\frac{3H}{4}]$.
Fix integer intervals $J_{x_t,y_t}$ and $U_{x_t,y_t}\subset \Z$ as in Proposition~\ref{mainprop}.
We set $I_t := [t-s_t+f^{(a)}(x_t), t-s_t+f^{(b)}(x_t) )$, where $a=\min J_{x_t,y_t}$, $b=\max J_{x_t,y_t}+1$.
Equivalently, by denoting by $t+I:=\{t+x \mid x\in I\}$ the sumset, we can write 
$I_t = t + \bigsqcup_{n\in J_{x_t,y_t}} I_t^n$, where $I_t^n := [-s_t+f^{(n)}(x_t),-s_t+f^{(n+1)}(x_t))$.
Since $\frac{H}{4}\leq t < \frac{3H}{4}$, $|J_{x_t,y_t}|\leq \frac{H}{5M}$ (by Claim~1) and $f\leq M$, we have $I_t\subset [0,H)$.
Also $t\in I_t$, because $0\in J_{x_t,y_t}$. 
For every $n\in J_{x_t,y_t}$ and $r\in I_t^n$ we have
\begin{align*}
(x_{t+r},s_{t+r})=T_{r}(x_t,s_t) &= (x_t+n\alpha, s_t+r-f^{(n)}(x_t)), \\
(y_{t+r},s'_{t+r})=T_{r}(y_t,s'_t) &= (y_t+m\alpha, s'_t+r-f^{(m)}(y_t)),
\end{align*}
where $f^{(m)}(y_t) \leq s'_t+r < f^{(m+1)}(y_t)$.
By estimating the distance between $(x_{t+r},s_{t+r})$ and $(y_{t+r},s'_{t+r})$ as in \eqref{dist} and
reasoning as in the proof of Claim~1 we can show that if $t+r\in C$, then nesessarily $m=n$.
Moreover, in this case $|s_t-s'_t|<\delta$ and $|s_{t+r}-s'_{t+r}|<\delta$, so that $|f^{(n)}(x_t)-f^{(n)}(y_t)|<2\delta< \frac{\beta L}{10^6M^2}$.
The latter inequality cannot hold if $n\in U_{x_t,y_t}$ (by Proposition~\ref{mainprop}\eqref{prop1c}).
In other words, if $n\in U_{x_t,y_t}$ then the corresponding interval $t+I_t^n$ is contained in $F$.
We can now estimate $\lambda(F\cap I_t) \geq L |U_{x_t,y_t}| > \frac{1}{10} L|J_{x_t,y_t}|$ by Proposition~\ref{mainprop}\eqref{prop1c}, while $\lambda(I_t) \leq M|J_{x_t,y_t}|$.
This proves Claim~2.

\smallskip

\textit{Final step.}
Let $F=\bigsqcup_{n} F_n$ where each $F_n$ is a (half-open) interval in $[0,H)$.
The family $\cI=\{I_t\}_{t\in C\cap[\frac{H}{4},\frac{3H}{4}]} \cup \{F_n\}_n$ covers $[\frac{H}{4},\frac{3H}{4}]$.
First, we can enlarge each interval by 1\% of its length to get open intervals and by compactness select a finite subfamily (of enlarged intervals) that covers $[\frac{H}{4},\frac{3H}{4}]$.
Then, by Vitali covering lemma, we can find a finite disjoint subfamily of $\cI$ of total measure at least $\frac{H}{7}$.
Since each interval in $\cI$ contains more than $\gamma$ proportion of $F$, $\lambda(F) > \frac{\gamma}{7}H$.
This means that inequality (\ref{hamineq}) fails if $\varepsilon < \frac{\gamma}{7}$.
The obtained contradiction finishes the proof of Theorem~\ref{mainth}.
\end{proof}

\subsection{Proof of Proposition~\ref{mainprop}}
Before we prove Proposition~\ref{mainprop} let us give an outline of how to prove \eqref{prop1b} and \eqref{prop1c}.
They are both a consequence of Lemma~\ref{lin}.
For $x,y\in \T$ satisfying \eqref{distanc}, if \eqref{lem3b} or \eqref{lem3c} in Lemma~\ref{lin} holds (that is forward or backward orbit of $[x,y]$ avoids the discontinuity long enough),
we define $J_{x,y}$ to be respectively $[0, \frac{C}{\|x-y\|}]\cap \Z$ or $[-\frac{C}{\|x-y\|},0]\cap \Z$, $C=C(\beta,M,L)$.
By the definition of $J_{x,y}$ and Lemma~\ref{c1}, for $n\in J_{x,y}$, $|f^{(n)}(x)-f^{(n)}(y)|=n(\text{ slope }+o(1))\|x-y\|$.
This means that the divergence is linear and hence \eqref{prop1b} and \eqref{prop1c} follow.
If $x,y\in \T$ satisfy \eqref{lem3a}, we define $J_{x,y}$ as $[0,2q_n-1]\cap \Z$.
In this case the divergence is given by hitting the discontinuity, but the number of times the orbit of $[x,y]$ hits $0$ is bounded by \eqref{dmleq}.
This gives \eqref{prop1b}.
And once hitting the discontinuity, the two points cannot come close for some time, which leads to \eqref{prop1c}.

The following lemma is classical.

\begin{lm}\label{c1} Let $g\in C^1(\T)$. Then
\begin{equation}\label{canc}
\lim_{|n|\to +\infty}\sup_{x,y\in\T, x\neq y}\frac{|g^{(n)}(x)-g^{(n)}(y)|}{|n|\|x-y\|}=0.
\end{equation}
\end{lm}
\begin{proof} Notice that for some $\theta^{(n)}_{x,y}\in\T$
$$
\frac{|g^{(n)}(x)-g^{(n)}(y)|}{|n|\|x-y\|}=
\left|\frac{1}{n}g'^{(n)}(\theta^{(n)}_{x,y})\right|
$$
and so the statement follows from the ergodic theorem and unique ergodicity of $R_\alpha$ since $g'\in C(\T)$ and $\int_\T g'd m_\T=0$.
\end{proof}

Let $(q_n)_{n\in\N}$ denote the sequence of denominators of $\alpha$, that is
$$
\frac{1}{2 q_n q_{n+1}} < \left| \alpha - \frac{p_n}{q_n} \right| < \frac{1}{q_n q_{n+1}},
$$
where
\[
\begin{array}{lll}
q_0 = 1, & q_1 = a_1, & q_{n} = a_{n}q_{n-1} + q_{n-2},\\
p_0 = 0, & p_1 = 1, & p_{n} = a_{n}p_{n-1} + p_{n-2}
\end{array}
\]
and $[0;a_1,a_2,\ldots]$ stands for the continued fraction expansion of $\alpha$ (see \cite{Kh}).
It follows that
$$
\frac{1}{2 q_{n+1}} < \| q_n\alpha \| < \frac{1}{q_{n+1}}.
$$
The reader can easily check that the partition of $\T$ by $0,\alpha,...,(q_n-1)\alpha$ has the form
$\{ R_\alpha^k I_n \mid 0\leq k < q_n - q_{n-1} \} \cup \{ R_\alpha^k I'_n \mid 0\leq k < q_{n-1} \}$,
where $I_n = [0,q_{n-1}\alpha]$ and $I'_n = [(-q_{n-1}+q_n)\alpha,0]\subset\T$.
(We recall that for $x,y\in\T$, $[x,y]$ stands for the shortest interval in $\T$ connecting $x$ and $y$.)
In particular,
\begin{equation}\label{pdiam}
\min_{0 \leq i < j < q_n} \|i\alpha - j\alpha\| = \|q_{n-1}\alpha\| > \frac{1}{2q_n}.
\end{equation}

\begin{lm}\label{lin}
Fix $x,y\in \T$, $x\neq y$, and let $n\in \N$ be any integer such that
\begin{equation}\label{distanc}
\|x-y\|< \frac{1}{6q_n}.
\end{equation}
Then one of the following holds:
\begin{enumerate}[(a)]
\item \label{lem3a} $0\in \bigcup_{k=0}^{q_n-1}R^k_\alpha[x,y]$;
\item \label{lem3b} $0\notin \bigcup_{k=0}^{\left[\frac{q_{n+1}}{6}\right]}R^k_\alpha[x,y]$;
\item \label{lem3c} $0\notin \bigcup_{k=-\left[\frac{q_{n+1}}{6}\right]}^{0}R^k_\alpha[x,y]$.
\end{enumerate}
\end{lm}

\begin{proof} Assume that \eqref{lem3a} does not hold.
Consider the partition $\cP_n$ of $\T$ by $0,-\alpha,...,(-q_n+1)\alpha$.
Let $i,j\in\{-q_n+1,...,0\}$ be unique such that
\begin{equation}\label{xysub}
[x,y]\subset [i\alpha,j\alpha]=I\in\cP_n.
\end{equation}
By \eqref{distanc} and \eqref{pdiam} it follows that either $d(i\alpha,[x,y])$ or $d(j\alpha,[x,y])$ is greater then $\frac{1}{6q_n}$.
Assume without loss of generality that $d(i\alpha,[x,y])> \frac{1}{6q_n}$ (the proof in the other case is analogous).
If $q_n\alpha - p_n<0$ we will show \eqref{lem3b}, and if $q_n\alpha - p_n>0$ we will show \eqref{lem3c}.
Let us conduct the proof assuming that $q_n\alpha-p_n<0$, the proof in the other case follows the same lines.
We need to show that for every $k\in\{-\left[\frac{q_{n+1}}{6}\right],...,0\}$,
\begin{equation}\label{atm}
k\alpha\notin [x,y].
\end{equation}
Each such $k\alpha$ belongs to a unique atom of $\cP_n$.
If $k\alpha\notin I$ then \eqref{atm} holds trivially by \eqref{xysub}.
All $k\in\{-\left[\frac{q_{n+1}}{6}\right],...,0\}$ for which $k\alpha\in I$ are of the form $k=i-m_kq_n$ for $m_k\leq \frac{q_{n+1}}{6q_n}$.
Therefore
$$
d(k\alpha,[x,y])\geq d(i\alpha,[x,y])-m_k\|q_n\alpha\|>\frac{1}{6q_n}-
\frac{q_{n+1}}{6q_n}\frac{1}{q_{n+1}}\geq 0.
$$
So \eqref{atm} also holds in this case. This finishes the proof of Lemma~\ref{lin}.
\end{proof}

Now we can prove Proposition~\ref{mainprop}.

\begin{proof}[Proof of Proposition~\ref{mainprop}]
Recall that $f=g+\{\cdot\}+c$, where $g\in C^1(\T)$, $0<L\leq 1 \leq M$, $0<\beta<1$.
Fix $\epsilon< \frac{\beta L}{10^6 M^2}$.
Let $\delta_0$ be such that for every $x,y\in \T$, $\|x-y\|< \delta_0$, and every $n\in \Z$
\begin{equation}\label{c1part}
|g^{(n)}(x)-g^{(n)}(y)|<\epsilon\max\{ 1, |n| \|x-y\| \}.
\end{equation}
The existence of such $\delta_0$ follows from Lemma~\ref{c1}.
Indeed, by Lemma~\ref{c1} there exists $n_\epsilon>0$ such that for all $n$, $|n|\geq n_\epsilon$, we have $|g^{(n)}(x)-g^{(n)}(y)|<\epsilon |n| \|x-y\|$.
It is enough to define $\delta_0:=\frac{\epsilon}{n_\epsilon \sup_\T |g'|}$.
We will also assume that
\begin{equation}\label{delta0}
\delta_0 \leq \frac{\beta L}{10^5 M^2}.
\end{equation}

Fix $x,y\in \T$, $\|x-y\|< \delta_0$.
We want to define $J_{x,y}$.
We will use Lemma~\ref{lin}.
Let $n\in \N$ be unique such that
\begin{equation}\label{dxy}
\frac{\beta}{200Mq_{n+1}}<\|x-y\|\leq \frac{\beta}{200Mq_n}.
\end{equation}
Then $n$ satisfies \eqref{distanc}, since $M\geq 1$.
Define $J_{x,y}$ as follows:
\begin{enumerate}[(A)]
\item \label{p1} if $x,y$ satisfy \eqref{lem3a}, set $J_{x,y}:=[0, 2q_n-1]\cap \Z$;
\item \label{p2} if $x,y$ satisfy \eqref{lem3b}, set $J_{x,y}:=[0, \frac{\beta L}{10^5M^{2}\|x-y\|}]\cap \Z$;
\item \label{p3} if $x,y$ satisfy \eqref{lem3c}, set $J_{x,y}:=[-\frac{\beta L}{10^5M^{2}\|x-y\|},0]\cap \Z$.
\end{enumerate}
We will show that $J_{x,y}$ defined above satisfies the assertions of Proposition~\ref{mainprop}.

\smallskip

\textit{Case \eqref{p1}.}
Notice that \eqref{prop1a} in Proposition~\ref{mainprop} holds trivially by \eqref{dxy} (since $2q_n\leq \frac{\beta}{100M\|x-y\|}$).
Let us show \eqref{prop1b}.
For $m\geq 0$ let
$$
D_m:=\left|\{i\in\{0,...,m-1\}\;:\; 0\in R_\alpha^i[x,y]\}\right|
$$
(the number of times the orbit of $[x,y]$ hits the discontinuity up to time $m$).
Notice that for every $n\in \frac{10M}{L}J_{x,y}\subset [0,\frac{20M}{L}q_n]$ by \eqref{dxy} and \eqref{pdiam} we have
\begin{equation}\label{dmleq}
D_n\leq \frac{40M}{L}.
\end{equation}
Therefore for every such $n$,
using \eqref{c1part} and \eqref{dxy}, we obtain
\begin{align*}
|f^{(n)}(x)-f^{(n)}(y)|
&\leq |g^{(n)}(x)-g^{(n)}(y)|+|\sum_{i=0}^{n-1}(\{x+i\alpha\}-\{y+i\alpha\})|\\
&< \max\{ 1, n \|x-y\| \} + n\|x-y\| + D_n \\
&< \frac{M}{L}+ \frac{M}{L}+\frac{40M}{L}\\
&< \frac{50M}{L}
\end{align*}
and this finishes the proof of \eqref{prop1b}.
For \eqref{prop1c} let $U_{x,y}:=[q_n,2q_n-1]\cap \Z$.
Obviously $|U_{x,y}| >\frac{1}{10}|J_{x,y}|$.
Since $x,y$ satisfy \eqref{lem3a}, it follows that for every $n\in U_{x,y}$ we have
$$
D_n \geq 1.
$$
Therefore for every $n\in U_{x,y}$ by \eqref{c1part} and \eqref{dxy}
$$
|f^{(n)}(x)-f^{(n)}(y)|
\geq D_n-n\|x-y\|-|g^{(n)}(x)-g^{(n)}(y)|
> 1-1/2
= 1/2
$$
and this completes the proof of \eqref{prop1c}.
So if $\eqref{lem3a}$ holds then Proposition~\ref{mainprop} holds for $x,y$.

\smallskip
It remains to conduct the proof in cases \eqref{p2} and \eqref{p3}.
The proofs in both cases are completely symmetric, one just needs to switch the time direction from positive (in case \eqref{p2}) to negative (in case \eqref{p3}).
Therefore we will present the proof in case \eqref{p2}, the proof in case \eqref{p3} follows the same lines.

\smallskip

\textit{Case \eqref{p2}.}
Notice that \eqref{prop1a} follows automatically by the definition of $J_{x,y}$.
Let us show \eqref{prop1b}.
Notice that $\frac{10M}{L}J_{x,y}\subset[0,\frac{\beta}{10^4M\|x-y\|}]\subset [0,\frac{q_{n+1}}{6}]$.
Hence by \eqref{lem3b} (which gives $D_n=0$), \eqref{c1part} and \eqref{dxy} we get for every $n\in \frac{10M}{L}J_{x,y}$
\begin{align*}
|f^{(n)}(x)-f^{(n)}(y)|
&\leq n\|x-y\|+|g^{(n)}(x)-g^{(n)}(y)|\\
&\leq n\|x-y\| + \max\{ 1, |n| \|x-y\| \}\\
&\leq \frac{\beta}{10^4M} + 1\\
&<2.
\end{align*}
This gives \eqref{prop1b}.
To get \eqref{prop1c} define $U_{x,y}:=[\frac{1}{2}\frac{\beta L}{10^5M^{2}\|x-y\|}, \frac{\beta L}{10^5M^{2}\|x-y\|}]\cap \Z$.
Then trivially $|U_{x,y}|\geq \frac{1}{2}|J_{x,y}|$ (notice that $U_{x,y}\neq\emptyset$ by \eqref{delta0}).
Moreover, for every $n\in U_{x,y}$,
$$
|f^{(n)}(x)-f^{(n)}(y)|
\geq n\|x-y\|-|g^{(n)}(x)-g^{(n)}(y)|
\geq \frac{1}{2}\frac{\beta L}{10^5M^{2}}-\epsilon
> \frac{\beta L}{10^6M^{2}}
$$
by the choice of $\epsilon$.
This gives \eqref{prop1c} and finishes the proof of Proposition~\ref{mainprop}
\end{proof}

\section{Concluding remarks}
It follows by \cite{FL04} that the spectral type of von Neumann flows is purely singular.
Whether or not the maximal spectral multiplicity of von Neumann flows is finite remains however an open problem.
It follows from our result that the popular method of estimating spectral multiplicity by the rank fails in this case.

It seems that the methods used for one discontinuity can be carried out for many discontinuities (however the proof becomes more subtle).

In a forthcoming paper \cite{KaSo} we will show that the slope is an isomorphism invariant for the von Neumann flows under consideration (with one discontinuity), i.e.\ flows are non-isomorphic when the slopes of the roof functions are different.
This would mean, in turn, that even when $\alpha$ is fixed, we have considered an uncountable family of pairwise non-isomorphic von Neumann flows.

\section{Acknowledgements}
The authors would like to thank K.~Fr\c{a}czek and M.~Lema\'{n}czyk for several remarks on the subject.
The first author would like to thank A.~Katok for discussions on spectral multiplicity of von Neumann flows.
We thank M.~Lema\'{n}czyk and C.~Ulcigrai for their comments on the final version of this paper.
The second author would like to thank ICTP, Trieste, where part of this work was done, for hospitality.
The research leading to these results has received funding from the European Research Council under the European Union's Seventh Framework Programme (FP/2007-2013) / ERC Grant Agreement n.~335989.

\bigskip

\textsc{Department of Mathematics, Pennsylvania State University
}\par\nopagebreak
\textit{E-mail address}:\: \texttt{adkanigowski@gmail.com}

\medskip

\textsc{School of Mathematics, University of Bristol
}\par\nopagebreak
\textit{E-mail address}:\: \texttt{solomko.anton@gmail.com}

\end{document}